\documentclass[11pt,a4paper,twocolumn]{article}
\usepackage[a4paper,margin=1.8cm]{geometry}
\usepackage{titlesec}
\titlespacing*{\section}{0pt}{2ex}{1ex}
\usepackage{amsmath,amssymb,amsthm}
\newtheorem{theorem}{Theorem}

\usepackage{amsthm}
\newtheorem{proposition}{Proposition}
\newtheorem*{remark*}{Remark}
\newtheorem{principle}[theorem]{Principle}

\usepackage[T1]{fontenc}
\usepackage[utf8]{inputenc}
\usepackage{lmodern}

\usepackage{abstract}

\date{}
\title{\textbf{First-Order Geometry, Spectral Compression, and Structural Compatibility under Bounded Computation}}
\author{
	\textbf{Changkai Li}\\
	\texttt{lck271828@gmail.com}\\
	\texttt{(Date:February,24,2026)}
}

\begin{document}
\twocolumn[
 \maketitle
 \vspace{-3 em}
 
 \begin{abstract}
 	\vspace{-1em}
 	Optimization under structural constraints is typically analyzed through projection or penalty methods, obscuring the geometric mechanism by which constraints shape admissible dynamics. We propose an operator-theoretic formulation in which computational or feasibility limitations are encoded by self-adjoint operators defining locally reachable subspaces. In this setting, the optimal first-order improvement direction emerges as a pseudoinverse-weighted gradient, revealing how constraints induce a distorted ascent geometry. We further demonstrate that effective dynamics concentrate along dominant spectral modes, yielding a principled notion of spectral compression, and establish a compatibility principle that characterizes the existence of common admissible directions across multiple objectives. The resulting framework unifies gradient projection, spectral truncation, and multi-objective feasibility within a single geometric structure.
 	
 \end{abstract}
 ]

\section{Introduction}

Many optimization and dynamical frameworks implicitly assume
unrestricted computational capacity.
Under such assumptions, the gradient
$\nabla J(\theta)$
directly determines the first-order admissible strategy direction.
However, when computational resources are bounded,
the set of admissible variations is constrained
by structural limitations on accessible directions.

The presence of computational constraints
raises a structural question:
how should first-order admissible strategy directions be characterized
when only a restricted subspace of variations is available?
While constrained optimization and related bounded models
have been studied in various contexts,
a unified geometric characterization of
computationally constrained first-order structure
remains incomplete.
In particular, three aspects require clarification:
(i) the intrinsic form of the constrained first-order admissible strategy direction,
(ii) the possibility of compressing this direction into
a finite structural representation,
and (iii) the compatibility of multiple structural constraints.

\textbf{Related Work:}
Optimization on smooth manifolds has been extensively studied in the context of Riemannian optimization \cite{absil2008optimization,boumal2023introduction}. 
Classical constrained first-order methods and projection-based approaches are treated in standard convex optimization literature \cite{boyd2004convex}. 
In parallel, bounded rationality and computational constraint models have been discussed in economic and decision-theoretic contexts \cite{simon1957models,gigerenzer2001bounded}. 
The present work differs in that it isolates a structural first-order geometry induced by computational accessibility and studies its spectral compression and multi-constraint compatibility properties.

Against this background, we now formalize the structural framework developed in this paper. The main contributions are as follows.

First, Theorem~1 establishes that
the unique first-order admissible strategy direction
under computational constraints
is given by the pseudoinverse-weighted gradient
\[
\Delta\theta^\star \propto H_C^\dagger(\theta)\nabla J(\theta).
\]

Second, Theorem~2 shows that this direction
admits a low-rank spectral approximation.
The resulting rule kernel
provides a compressed representation
with explicit residual bounds.

Third, Principle~3 introduces
a structural compatibility threshold
characterized by a critical coupling parameter
governing the existence of a common admissible strategy tangent direction.

Together, these results provide
a three-layer structural decomposition
of bounded computational systems:
constrained first-order geometry,
spectral rule compression,
and multi-constraint compatibility.

The remainder of the paper is organized as follows.
Section~2 presents the problem statement and the minimal geometric setting.
Section~3 develops the geometric structure induced by bounded computation.
Section~4 establishes the first-order constrained optimal direction (Theorem~1).
Section~5 proves the low-rank rule-kernel approximation result (Theorem~2).
Section~6 formulates the structural compatibility threshold principle (Principle~3).
Section~7 concludes with a minimal discussion summarizing the structural framework.

\section{Problem Statement and Minimal Setting}

\subsection{Multi-Agent Strategy Structure}

Let $\mathcal{A}$ denote a finite set of agents.

Each agent selects strategies from a common strategy manifold $\Theta$, assumed to be a smooth finite-dimensional manifold endowed with a Riemannian metric. The metric induces an inner product on each tangent space $T_\theta \Theta$.

A collective strategy profile is represented by a point on the strategy manifold
\[
\theta \in \Theta.
\]
Here $\theta$ denotes the strategy itself (in local coordinates on $\Theta$), rather than a preference-dependent parameterization.

We assume the existence of a payoff functional
\[
J : \Theta \rightarrow \mathbb{R},
\]
which is continuously differentiable.

\subsection{Complexity Constraint and Computational Budget}

Each strategy profile $\theta$ incurs a complexity cost measured by a functional
\[
C : \Theta \rightarrow \mathbb{R}_{\ge 0},
\]
assumed to be continuously differentiable.

The functional $C$ encodes computational cost of representing or evaluating strategies and does not modify the payoff functional $J$.

The system operates under a fixed computational budget
\[
C(\theta) \leq \kappa,
\]
where $\kappa > 0$ is a constant.

The feasible set is defined as
\[
\Theta_\kappa := \{ \theta \in \Theta \mid C(\theta) \le \kappa \}.
\]

\subsection{Local Feasibility Structure}

We consider local perturbations $\Delta \theta \in T_\theta \Theta$ such that
\[
\theta + \Delta \theta \in \Theta_\kappa.
\]

For sufficiently small perturbations, first-order feasibility requires
\[
\nabla C(\theta) \cdot \Delta \theta \le 0,
\]
where the gradient is defined with respect to the Riemannian metric.

This condition defines the local feasible tangent cone at $\theta$.

\subsection{Problem Formulation}

Given the continuously differentiable functionals $J$ and $C$ under the fixed budget $\kappa$, we seek to characterize the first-order structure of first-order admissible strategy directions in $\Theta_\kappa$.

In particular, the problem is to determine whether there exists a canonical first-order \emph{strategy variation} direction determined jointly by the gradient structure of $J$ and the computational accessibility geometry induced by bounded computation.

Throughout this paper, all first-order variations are understood as variations of strategies on the strategy manifold, rather than updates of an external parameterization.
\section{Geometric Structure of Bounded Computation}
 
 \subsection{Reachable Direction Structure}
 
 For each strategy profile $\theta \in \Theta$, we associate a computationally reachable subspace
 \[
 \mathcal{V}_\theta \subset T_\theta \Theta.
 \]
 
 We assume:
 
 \begin{itemize}
  \item[(A1)] $\mathcal{V}_\theta$ is a linear subspace of $T_\theta \Theta$.
  \item[(A2)] $\mathcal{V}_\theta$ is closed with respect to the Riemannian inner product.
  \item[(A3)] The assignment $\theta \mapsto \mathcal{V}_\theta$ is measurable.
 \end{itemize}
 
 \subsection{Computational Geometry Operator}
 
 Bounded computation does not merely restrict admissible strategy tangent directions; it induces a geometric distortion within the reachable subspace.
 
 We therefore define a constraint-induced linear operator
 \[
 H_C(\theta) : T_\theta \Theta \to T_\theta \Theta
 \]
 satisfying:
 
 \begin{itemize}
  \item[(B1)] $H_C(\theta)$ is self-adjoint.
  \item[(B2)] $H_C(\theta)$ is positive semidefinite.
  \item[(B3)] $\operatorname{Im}(H_C(\theta)) = \mathcal{V}_\theta$.
  \item[(B4)] $\ker(H_C(\theta)) = \mathcal{V}_\theta^\perp$.
 \end{itemize}
 
 Thus $H_C(\theta)$ encodes both accessibility and directional weighting induced by bounded computation.
 
 When $H_C(\theta)$ reduces to the orthogonal projection onto $\mathcal{V}_\theta$, the geometry is isotropic within the reachable subspace. In general, however, $H_C(\theta)$ may exhibit nontrivial spectral structure.
 
 \subsection{Pseudoinverse Structure}
 
 Since $H_C(\theta)$ may be rank-deficient, we define its Moore–Penrose pseudoinverse
 \[
 H_C^\dagger(\theta).
 \]
 
 The operator $H_C^\dagger(\theta)$ acts as the generalized inverse within the reachable subspace and vanishes on its orthogonal complement.
 
 The pair $(H_C(\theta), H_C^\dagger(\theta))$ defines the first-order computational geometry at strategy profile $\theta$.
 
 \subsection{Decomposition of Tangent Space}
 
 The tangent space admits the orthogonal decomposition
 \[
 T_\theta \Theta = \mathcal{V}_\theta \oplus \mathcal{V}_\theta^\perp,
 \]
 with $H_C(\theta)$ acting nontrivially only on $\mathcal{V}_\theta$.
 
 This structure determines the admissible first-order variation geometry under bounded computation.
 
 \section{Theorem 1: First-Order Constrained Optimal Direction}
 
 \subsection{Computational Effort Geometry}
 
 Fix $\theta \in \Theta$.  
 Let $H_C(\theta): T_\theta\Theta \to T_\theta\Theta$ be the self-adjoint, positive semidefinite operator defined in Section 2,
 with
 \[
 \operatorname{Im}(H_C(\theta))=\mathcal{V}_\theta,
 \qquad
 \ker(H_C(\theta))=\mathcal{V}_\theta^\perp.
 \]
 
 Define the computational effort functional
 \begin{equation}
  \mathcal{E}_\theta(\Delta\theta)
  :=
  \langle \Delta\theta,\, H_C(\theta)\Delta\theta\rangle_\theta .
 \end{equation}
 
 We restrict admissible first-order variations to the reachable subspace:
 \begin{equation}
  \mathcal{D}_\theta
  :=
  \left\{
  \Delta\theta \in \mathcal{V}_\theta
  \;\middle|\;
  \mathcal{E}_\theta(\Delta\theta)=1
  \right\}.
 \end{equation}
 
 \subsection{First-Order Admissible Strategy Direction Problem}
 
 Assume $J:\Theta\to\mathbb{R}$ is differentiable at $\theta$.
 Let $g := \nabla J(\theta)$.
 
 We consider the variational problem
 \begin{equation}
  \label{eq:varproblem}
  \max_{\Delta\theta\in\mathcal{D}_\theta}
  \;
  \langle g,\Delta\theta\rangle_\theta .
 \end{equation}
 
 \subsection{Main Result}
 
 \begin{theorem}[First-Order Constrained Optimal Direction]
  \label{thm:firstorder}
  Fix $\theta\in\Theta$ and assume $J$ is differentiable at $\theta$.
  Let $H_C(\theta)$ satisfy the properties in Section 2 and let $H_C^\dagger(\theta)$ denote its Moore--Penrose pseudoinverse.
  
  If $H_C(\theta) g \neq 0$, then the maximizers of \eqref{eq:varproblem}
  are exactly the rays generated by
  \begin{equation}
   \label{eq:optimaldirection}
   \Delta\theta^\star
   \propto
   H_C^\dagger(\theta)\, g,
  \end{equation}
  normalized so that $\mathcal{E}_\theta(\Delta\theta^\star)=1$.
  
  If $H_C(\theta) g = 0$, then
  \[
  \langle g,\Delta\theta\rangle_\theta = 0
  \quad
  \text{for all } \Delta\theta\in\mathcal{D}_\theta,
  \]
  and no admissible strategy tangent direction yields positive first-order payoff under the computational constraint.
 \end{theorem}
 
 \subsection{Proof}
 
 \begin{proof}
  Fix $\theta$ and abbreviate
  $H_C := H_C(\theta)$,
  $H_C^\dagger := H_C^\dagger(\theta)$,
  $g := \nabla J(\theta)$.
  
  Since $\mathcal{D}_\theta \subset \mathcal{V}_\theta = \operatorname{Im}(H_C)$,
  we may restrict all variations to $\operatorname{Im}(H_C)$.
  
  Because $H_C$ is self-adjoint,
  the operator $H_C H_C^\dagger$ is the orthogonal projection onto $\operatorname{Im}(H_C)$.
  Therefore for any $\Delta\theta\in\mathcal{V}_\theta$,
  \begin{equation}
   \langle g,\Delta\theta\rangle_\theta
   =
   \langle H_C H_C^\dagger g,\Delta\theta\rangle_\theta
   =
   \langle H_C^\dagger g,\; H_C \Delta\theta\rangle_\theta.
  \end{equation}
  
  Define the semi-inner product
  \[
  \langle u,v\rangle_{H_C}
  :=
  \langle u,\; H_C v\rangle_\theta.
  \]
  
  On $\operatorname{Im}(H_C)$ this form is non-degenerate.
  Indeed, if $v\in\operatorname{Im}(H_C)$ and
  $\langle v,H_C v\rangle_\theta=0$,
  then $v\in\ker(H_C)\cap\operatorname{Im}(H_C)$,
  which implies $v=0$ since
  $\ker(H_C)=\operatorname{Im}(H_C)^\perp$.
  
  Thus on $\operatorname{Im}(H_C)$ the norm
  \[
  \|v\|_{H_C}^2 := \langle v,H_C v\rangle_\theta
  \]
  is well-defined and coincides with $\mathcal{E}_\theta(v)$.
  
  For $\Delta\theta\in\mathcal{D}_\theta$ we have $\|\Delta\theta\|_{H_C}=1$.
  Applying Cauchy--Schwarz in $\langle\cdot,\cdot\rangle_{H_C}$,
  \[
  \langle g,\Delta\theta\rangle_\theta
  =
  \langle H_C^\dagger g,\Delta\theta\rangle_{H_C}
  \le
  \|H_C^\dagger g\|_{H_C}\,\|\Delta\theta\|_{H_C}
  =
  \|H_C^\dagger g\|_{H_C}.
  \]
  
  Equality holds if and only if $\Delta\theta$
  is positively colinear with $H_C^\dagger g$
  in $\operatorname{Im}(H_C)$.
  
  Therefore, if $H_C g\neq 0$
  (equivalently $H_C^\dagger g\neq 0$),
  the maximizers are exactly the rays generated by $H_C^\dagger g$,
  normalized to unit computational effort.
  
  If $H_C g = 0$,
  then $g\in\ker(H_C)$,
  and since $\Delta\theta\in\operatorname{Im}(H_C)$,
  \[
  \langle g,\Delta\theta\rangle_\theta = 0
  \]
  for all admissible strategy tangent directions,
  the first-order payoff variation is non-positive.
 \end{proof}
 
 \section{Theorem 2: Rule Kernel Approximation}
 
 \subsection{Spectral Structure of the Computational Operator}
 
 Fix $\theta \in \Theta$.
 Since $H_C(\theta)$ is self-adjoint and positive semidefinite,
 it admits the spectral decomposition
 \[
 H_C(\theta)
 =
 \sum_{i=1}^{r} \lambda_i \, u_i \otimes u_i^T,
 \]
 where:
 
 \begin{itemize}
  \item $\lambda_1 \ge \lambda_2 \ge \dots \ge \lambda_r > 0$ are the positive eigenvalues,
  \item $\{u_1,\dots,u_r\}$ form an orthonormal basis of $\operatorname{Im}(H_C(\theta))$,
  \item $\ker(H_C(\theta))$ corresponds to eigenvalue $0$.
 \end{itemize}
 
 The Moore–Penrose pseudoinverse is therefore
 \[
 H_C^\dagger(\theta)
 =
 \sum_{i=1}^{r}
 \frac{1}{\lambda_i}\,
 u_i \otimes u_i^T.
 \]
 
 \subsection{Definition of the Rule Kernel}
 
 For any $k < r$, define the rank-$k$ truncated operator
 \[
 H_{C,k}^\dagger(\theta)
 :=
 \sum_{i=1}^{k}
 \frac{1}{\lambda_i}\,
 u_i \otimes u_i^T.
 \]
 
 We refer to $H_{C,k}^\dagger(\theta)$ as the
 \emph{rank-$k$ rule kernel} associated with $H_C(\theta)$.
 
 \subsection{Main Result}
 
 \begin{theorem}[Low-Rank Approximation of First-Order Direction]
  \label{thm:rulekernel}
  Let $g = \nabla J(\theta)$.
  The first-order constrained optimal direction
  \[
  \Delta\theta^\star
  =
  H_C^\dagger(\theta)\, g
  \]
  admits the decomposition
  \[
  \Delta\theta^\star
  =
  H_{C,k}^\dagger(\theta)\, g
  +
  R_k(g),
  \]
  where the residual satisfies
  \begin{equation}
   \|R_k(g)\|^2
   =
   \sum_{i=k+1}^{r}
   \frac{1}{\lambda_i^2}
   |\langle g,u_i\rangle|^2.
  \end{equation}
  
  In particular, the operator-norm error satisfies
  \[
  \| H_C^\dagger(\theta) - H_{C,k}^\dagger(\theta) \|_{\mathrm{op}}
  =
  \frac{1}{\lambda_{k+1}}.
  \]
 \end{theorem}
 
 \subsection{Proof}
 
 \begin{proof}
  Fix $\theta$ and abbreviate $H_C := H_C(\theta)$.
  Let $g = \nabla J(\theta)$.
  
  Using the spectral representation of $H_C^\dagger$,
  \[
  H_C^\dagger g
  =
  \sum_{i=1}^{r}
  \frac{1}{\lambda_i}
  \langle g,u_i\rangle u_i.
  \]
  
  Similarly,
  \[
  H_{C,k}^\dagger g
  =
  \sum_{i=1}^{k}
  \frac{1}{\lambda_i}
  \langle g,u_i\rangle u_i.
  \]
  
  Therefore the residual equals
  \[
  R_k(g)
  =
  \sum_{i=k+1}^{r}
  \frac{1}{\lambda_i}
  \langle g,u_i\rangle u_i,
  \]
  and its squared norm is
  \[
  \|R_k(g)\|^2
  =
  \sum_{i=k+1}^{r}
  \frac{1}{\lambda_i^2}
  |\langle g,u_i\rangle|^2.
  \]
  
  Since the largest singular value of
  $H_C^\dagger - H_{C,k}^\dagger$
  is $\frac{1}{\lambda_{k+1}}$,
  the operator-norm identity follows immediately.
 \end{proof}
 
\section{Principle 3: Structural Compatibility Threshold}

\subsection{Family of Admissible Rule Cones}

Fix $\theta \in \Theta$.

Let $\{ \mathcal{K}_i \}_{i=1}^m$ be a finite family of closed convex cones in 
$\mathcal{V}_\theta \subset T_\theta\Theta$.
Assume:

\begin{itemize}
 \item[(C1)] Each $\mathcal{K}_i$ is closed and convex.
 \item[(C2)] $\mathcal{K}_i \subset \mathcal{V}_\theta$.
 \item[(C3)] $\mathcal{K}_i$ is nonempty.
\end{itemize}

\subsection{Coupling Family}

Let $\gamma \ge 0$.
A coupling family is a collection of maps
\[
L_\gamma : \mathcal{V}_\theta \to \mathcal{V}_\theta.
\]
For any set $A\subseteq \mathcal{V}_\theta$, we use the convention
\[
L_\gamma(A) := \{ L_\gamma(x) : x\in A\}.
\]

Assume:

\begin{itemize}
 \item[(D1)] $L_0 = \mathrm{Id}$.
 \item[(D2)] If $\gamma_1 \le \gamma_2$, then
 \[
 L_{\gamma_1}(\mathcal{K}_i)
 \subseteq
 L_{\gamma_2}(\mathcal{K}_i)
 \quad \text{for all } i.
 \]
 \item[(D3)] For each $i$ and $\gamma$, $L_\gamma(\mathcal{K}_i)$ is a closed convex cone.
\end{itemize}

Define the $\gamma$-coupled intersection:
\[
\mathcal{K}_{\mathrm{int}}(\gamma)
:=
\bigcap_{i=1}^m
L_\gamma(\mathcal{K}_i).
\]

\subsection{Compatibility Functional}

Define
\[
\Phi(\gamma)
:=
S\big(\mathcal{K}_{\mathrm{int}}(\gamma)\cap \mathbb{S}\big),
\]
where $\mathbb{S}$ is the unit sphere in $\mathcal{V}_\theta$ and $S(\cdot)$ is the spherical measure
(with $S(\emptyset)=0$).

\subsection{Monotonicity}

\begin{proposition}
 If $\gamma_1 \le \gamma_2$, then
 \[
 \Phi(\gamma_1)
 \le
 \Phi(\gamma_2).
 \]
\end{proposition}

\begin{proof}
 By (D2),
 \[
 \mathcal{K}_{\mathrm{int}}(\gamma_1)
 \subseteq
 \mathcal{K}_{\mathrm{int}}(\gamma_2).
 \]
 Intersecting with $\mathbb{S}$ preserves inclusion. Since $S$ is monotone under inclusion, the claim follows.
\end{proof}

\subsection{Compatibility Threshold}

Define
\[
\gamma^*
:=
\inf
\left\{
\gamma \ge 0 :
\mathcal{K}_{\mathrm{int}}(\gamma) \neq \emptyset
\right\}.
\]

\begin{principle}[Structural Compatibility Threshold]
 For $\gamma < \gamma^*$, one has $\mathcal{K}_{\mathrm{int}}(\gamma)=\emptyset$.
 For $\gamma > \gamma^*$, one has $\mathcal{K}_{\mathrm{int}}(\gamma)\neq\emptyset$.
\end{principle}

\begin{remark*}
 If the set $\{\gamma\ge 0:\mathcal{K}_{\mathrm{int}}(\gamma)\neq\emptyset\}$ is closed,
 then $\mathcal{K}_{\mathrm{int}}(\gamma^*)\neq\emptyset$ and the second statement can be strengthened to
 $\gamma \ge \gamma^*$.
\end{remark*}

\section{Minimal Discussion}

We summarize the structural results established in this paper.

First, under bounded computational geometry,
Theorem 1 characterizes the unique first-order admissible strategy direction
as the pseudoinverse-weighted gradient
\[
\Delta\theta^\star \propto H_C^\dagger(\theta)\nabla J(\theta).
\]

Second, Theorem 2 shows that this direction admits
a low-rank spectral approximation.
The rule kernel $H_{C,k}^\dagger$ provides
a compressed representation of the constrained first-order strategy structure,
with explicit residual bounds.

Third, Principle 3 introduces a structural compatibility threshold
for families of admissible rule cones.
The parameter $\gamma$ governs the enlargement of constraint sets,
and the threshold $\gamma^*$ characterizes
the minimal structural coupling required
for the existence of a common admissible strategy tangent direction.

Taken together,
these results define a three-layer structural decomposition:

\begin{itemize}
 \item constrained first-order geometry,
 \item spectral rule compression,
 \item multi-constraint compatibility.
\end{itemize}

The present paper focuses exclusively on
the structural properties of bounded computational systems.
Further extensions may incorporate
dynamic evolution or semantic interpretation,
but such developments lie beyond the scope of the current structural analysis.
 
 
\bibliographystyle{plain}
\bibliography{references}
\end{document}